\RequirePackage{fix-cm}
\documentclass{svjour3}                     % onecolumn (standard format)
\smartqed  % flush right qed marks, e.g. at end of proof
\usepackage{graphicx}
\usepackage{amsmath}
\usepackage{amsfonts}
\usepackage{amssymb}
\usepackage{graphicx}

\newcommand{\norm}[1]{\left\Vert#1\right\Vert}
\newcommand{\origin}{o}

\newcommand{\mball}{B_X}
\newcommand{\Mspace}{X}
\newcommand{\diam}[1]{\delta\left(#1\right)}
\newcommand{\thick}[1]{\Delta\left(#1\right)}

\newtheorem{prop}[section]{Proposition}

\begin{document}

\title{Complete sets need not be reduced in Minkowski
  spaces\thanks{Senlin Wu is partially supported by the National
    Natural Science Foundation of China (grant numbers 11371114 and
    11171082) and by Deutscher Akademischer Austauschdienst.}}



\author{Horst Martini \and Senlin Wu}


\institute{Horst Martini \at
              Faculty of Mathematics, Chemnitz University of Technology, 09107 Chemnitz, Germany\\
              \email{horst.martini@mathematik.tu-chemnitz.de}           %  \\
           \and
           Senlin Wu \at
           Department of Applied Mathematics, Harbin University of
           Science and Technology, 150080 Harbin, China\\
              \email{wusenlin@outlook.com}           %  \\
}

\date{Received: date / Accepted: date}

\maketitle

\begin{abstract}
  It is well known that in $n$-dimensional Euclidean space ($n\geq 2$)
  the classes of (diametrically) complete sets and of bodies of
  constant width coincide. Due to this, they both form a proper
  subfamily of the class of reduced bodies. For $n$-dimensional
  Minkowski spaces, this coincidence is no longer true if $n\geq
  3$.
  Thus, the question occurs whether for $n\geq 3$ any complete set is
  reduced. Answering this in the negative for $n\geq 3$, we
  construct $(2^{k}-1)$-dimensional ($k\geq 2$) complete sets which are not reduced.
\end{abstract}

\section{Introduction}
\label{sec:introduction}
We denote by $\Mspace=(\mathbb{R}^n, \norm{\cdot})$ an $n$-dimensional
normed or \emph{Minkowski space}, i.e., an $n$-dimensional real Banach
space with origin $\origin$. We will use the common abbreviations
${\rm aff}$ and ${\rm conv}$ for affine and convex hull, respectively. A \emph{convex body} in $\Mspace$ is a compact, convex
set with interior points, and the norm $\norm{\cdot}$ is induced by the
$\origin$-symmetric convex body $\mball$, called the \emph{unit ball}
of $\Mspace$, via 
\begin{displaymath}
  \norm{x}=\min\{\lambda\geq 0:~x\in\lambda\mball\},~\forall x\in\mathbb{R}^{n}.
\end{displaymath}
An interesting part of the geometry of Minkowski spaces (see \cite{Tho}, \cite{M-S-W}, and \cite{M-S}) refers to natural extensions of
notions from classical convexity to normed spaces. Such notions
are, e.g., special classes of convex bodies, like those of constant width, complete sets, and reduced bodies. It is well known
that also in Minkowski spaces the \emph{width} of a convex body $K$ in
direction $u$ is the (Minkowskian) distance of the two supporting
hyperplanes $H_1, H_2$ of $K$ which are both parallel to the
$(n-1)$-dimensional subspace orthogonal (in the Euclidean sense) to $u$. The minimum and the maximum of the
\emph{width function} of $K$ are called the \emph{thickness} $\thick{K}$
and the \emph{diameter} $\diam{K}$ of $K$, respectively. In any Minkowski
space $\Mspace$, \emph{bodies of constant width} are precisely determined by
the condition $\thick{K}=\diam{K}$, and a convex body is called
\emph{(diametrically) complete} in $\Mspace$ if it is not properly contained
in a compact set of the same diameter. In a sense dually, a convex
body $R$ is said to be \emph{reduced} in $\Mspace$ if $\thick{C}<\thick{R}$
holds for any convex body $C$ properly contained in $R$. There are
interesting open problems about complete and reduced bodies in
Minkowski spaces (see the survey \cite{L-M2}), and even for the Euclidean
norm surprisingly elementary questions are still unsolved
(cf. \cite{L-M1}). E.g., it is not known whether in Euclidean $n$-space
$E^n$ ($n\geq 3$) there are reduced convex $n$-polytopes.

It is also interesting that the notion of reduced body, although it
somehow dualizes that of completeness, yields a proper superset of the
class of complete sets in $E^n$. The reason is that in $E^n$ the notions
of constant width and completeness coincide, and that one can easily construct
reduced bodies that are not of constant width. In Minkowski spaces,  
this coincidence
of constant width and completeness is only assured for $n = 2$  
(Meissner's theorem; see p. 98 in \cite{M-S}),
and therefore in any Minkowski plane the class of reduced bodies  
clearly contains that of complete sets.
Thus, the question remains whether in $n$-dimensional Minkowski spaces  
($n\geq 3$)
any complete set is reduced. A wrong short formulation was given in the survey \cite{L-M2} (cf. line 10 of the Abstract and, in  
repeated form, line 13 of page 406), which would imply that in  
Minkowski spaces any complete set is reduced. To correct this, we want to construct non-planar complete sets that are not reduced.

\section{Results}
\label{sec:results}
Let $\Mspace$ be the Minkowski space $l_{1}^{3}$, i.e., the Banach space
on $\mathbb{R}^{3}$ endowed with the taxicab norm. Then $\mball$ is the crosspolytope whose vertices
are
\begin{displaymath}
  \pm(1,0,0),~\pm(0,1,0),\text{ and }\pm(0,0,1).
\end{displaymath}
Put $K$ to be the convex hull of the following
four points (see Figure \ref{fig:ex1}):
\begin{displaymath}
  a_{1}=(-1,-1,-1),~a_{2}=(1,1,-1),~a_{3}=(1,-1,1),\text{ and }a_{4}=(-1,1,1).
\end{displaymath}

\begin{figure}
  \centering
\includegraphics{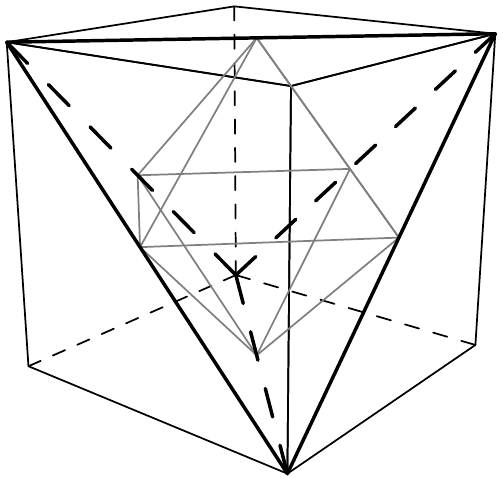}
  \caption{The tetrahedron $K$ is not reduced.}
  \label{fig:ex1}
\end{figure}

Clearly, 
\begin{equation}
  \label{eq:1}
  \norm{a_{i}-a_{j}}=4,~\forall\{i,j\}\subseteq\{1,2,3,4\}.
\end{equation}
\par
\emph{Claim 1: The tetrahedron $K$ is a complete set whose diameter is $4$.}
\begin{proof}
  We only need to show that, for each boundary point $x$ of $K$, there
  exists a point $y\in K$ such that $\norm{x-y}=\diam{K}$, which
  follows directly from \eqref{eq:1} and the trivial fact that
  \begin{displaymath}
    \norm{a_{i}-\frac{1}{3}(a_{j}+a_{k}+a_{l})}=4,~\forall\{i,j,k,l\}=\{1,2,3,4\}.
  \end{displaymath}\qed
\end{proof}

\emph{Claim 2: $\thick{K}=2$.}
\begin{proof}
From the equalities
\begin{gather*}
  \frac{1}{2}(a_{2}+a_{3})=(1,0,0),~\frac{1}{2}(a_{1}+a_{2})=(0,0,-1),\\
\frac{1}{2}(a_{2}+a_{4})=(0,1,0),~\frac{1}{2}(a_{1}+a_{3})=(0,-1,0),\\
  \frac{1}{2}(a_{3}+a_{4})=(0,0,1),~\frac{1}{2}(a_{1}+a_{4})=(-1,0,0),
\end{gather*}
it follows that $B_{X}\subseteq K$. Therefore $\thick{K}\geq 2$. Since
the hyperplanes
\begin{displaymath}
\{(\alpha,\beta,\gamma):~\gamma=1\}~\text{and}~\{(\alpha,\beta,\gamma):~\gamma=-1\}   
\end{displaymath}
are two parallel supporting
hyperplanes of $K$ and the distance between them is $2$,
$\thick{K}\leq 2$. Therefore, $\thick{K}=2$.\qed
\end{proof}

Now we are ready to show that $K$ is not reduced. Namely, the
thickness of the convex body
\begin{displaymath}
  K\cap\{(\alpha,\beta,\gamma):~\alpha+\beta+\gamma\geq -1\},
\end{displaymath}
which is a proper subset of $K$, is still $2$.

\bigskip

Now we try to extend this example to higher dimensions. Our
main tool is that of Walsh matrices $\mathbf{H}(2^{n})$
(cf. \cite[Chapters I and II]{Yang} and \cite{S-Y}) which can be
constructed by the following procedure:
\begin{gather*}
  \mathbf{H}(2^{1})=
  \begin{pmatrix}
    1&1\\
1&-1
  \end{pmatrix},
\\
\mathbf{H}(2^{2})=
\begin{pmatrix}
  1&1&1&1\\
1&-1&1&-1\\
1&1&-1&-1\\
1&-1&-1&1
\end{pmatrix},\\
\mathbf{H}(2^{k})=
\begin{pmatrix}
  H(2^{k-1})&H(2^{k-1})\\
H(2^{k-1})&-H(2^{k-1})
\end{pmatrix},~\forall k\in\mathbb{N}, k\geq 2.
\end{gather*}

Walsh matrices and the related Hadamard matrices can be a helpful tool
for modelling and solving geometric problems in Banach spaces and high
dimensions (see, e.g., \cite{Sch} and \cite[p. 51]{Zong}). Since any
Walsh matrix is a Hadamard matrix, we have
\begin{equation}\label{eq:Hadamard}
  \mathbf{H}(2^{n})\mathbf{H}(2^{n})^{\rm
    T}=2^{n}\mathbf{I}_{2^{n}},~\forall n\in\mathbb{N}, n\geq 1.
\end{equation}
\par
For each integer $n\geq 2$, let $(\alpha_{i,j})_{0\leq i,j\leq
  2^{n}-1}$ be the matrix $\mathbf{H}(2^{n})$. For each integer $0\leq
i\leq 2^{n}-1$, put 
\begin{displaymath}
b_{i}=(\alpha_{i,j})_{j=0}^{2^{n}-1}~\text{and}~ a_{i}=(\alpha_{i,j})_{j=1}^{2^{n}-1}.   
\end{displaymath}
Then
\begin{displaymath}
  V=\{a_{i}:~0\leq i\leq 2^{n}-1\}
\end{displaymath}
is a set of $2^{n}$ points in $\mathbb{R}^{2^{n}-1}$. By
\eqref{eq:Hadamard}, we have that
\begin{displaymath}
\{b_{i}:~0\leq i\leq 2^{n}-1\}
\end{displaymath}
is linearly independent and, therefore, $V$ is affinely
independent. Thus $S={\rm conv}V$ is a simplex in
$\mathbb{R}^{2^{n}-1}$.

\begin{prop}\label{prop:main}
  Let $\Mspace=l_{1}^{2^{n}-1}$, and $\norm{\cdot}$ be the norm of
  $\Mspace$. Then
  \begin{enumerate}
  \item $\diam{S}=2^{n}$;
  \item for each vertex $v$ of $S$ and each point $w$ in the facet of
    $S$ opposite to $v$, $\norm{v-w}=2^{n}$; moreover, the distance
    from $v$ to ${\rm aff}F_{v}$ is $2^{n}$;
  \item $\mball\subseteq S$;
  \item $\thick{S}=2$;
  \item $S$ is not reduced.
  \end{enumerate}
\end{prop}
\begin{proof}
  1. By \eqref{eq:Hadamard}, for each set
  $\{i,j\}\subset\{k:~k\in\mathbb{Z}, 0\leq k\leq 2^{n}-1\}$ the
  inner product of $b_{i}$ and $b_{j}$ is $0$. Therefore $b_{i}$ and
  $b_{j}$ are different in precisely $2^{n-1}$ coordinates. This
  implies that
  \begin{displaymath}
    \norm{a_{i}-a_{j}}=2^{n}.
  \end{displaymath}
Thus, the distance between each pair of the vertices of $S$ is
$2^{n}$, which implies that $\diam{S}=2^{n}$.
\par
2. Being a Hadamard
matrix, $\mathbf{H}(2^{n})$ must have the following property: each
column of $\mathbf{H}(2^{n})$, except for
the first one, is evenly divided between $1$ and $-1$. Thus
\begin{displaymath}
  \sum\limits_{i=0}^{2^{n}-1}a_{i}=\origin.
\end{displaymath}
Let $v$ be an arbitrary
vertex of $S$, and $F_{v}$ be the facet of $S$ opposite to $v$. Then
$F_{v}={\rm conv}(V\setminus\{v\})$, and the point
\begin{displaymath}
  w=\frac{1}{2^{n}-1}\sum\limits_{z\in V\setminus\{v\}}^{}z
\end{displaymath}
is a relatively interior point of $F_{v}$. Moreover,
\begin{align*}
  \norm{v-w}&=\norm{v-\frac{1}{2^{n}-1}\sum\limits_{z\in
              V\setminus\{v\}}^{}z}\\
&=\norm{v+\frac{1}{2^{n}-1}v}\\
&=2^{n}.
\end{align*}
Since the functional
\begin{align*}
  f:~F_{v}&\mapsto \mathbb{R}\\
z&\mapsto \norm{v-z}
\end{align*}
is convex, 
\begin{displaymath}
  \norm{v-z}=2^{n},~\forall z\in F_{v}.
\end{displaymath}
The second part of the result follows from the fact that the functional
\begin{align*}
  f:~{\rm aff}F_{v}&\mapsto \mathbb{R}\\
z&\mapsto \norm{v-z}
\end{align*}
is convex and constant in $F_{v}$.
\par
3. For each vertex $v$ of $S$, denote by $H_{v}$ the supporting
halfspace of $S$ determined by $F_{v}$. Then 
\begin{displaymath}
  S=\bigcap\limits_{v\in V}^{}H_{v}.
\end{displaymath}
Thus we only need to show that, for each $v\in V$, $H_{v}$ is also a
supporting halfspace of $\mball$. Since $\origin\in S$, as can be easily verified, we only need to show that each hyperplane determined
by $F_{v}$, which is ${\rm aff}F_{v}$, is a supporting hyperplane of $\mball$ at the point
$-(1/(2^{n}-1))v$. Clearly,
\begin{displaymath}
  -\frac{1}{2^{n}-1}v\in F_{v}\cap \mball.
\end{displaymath}
We only deal with the case when $v=a_{0}$; the other cases can be
dealt with in a similar way. For each point $w\in {\rm aff}F_{v}$,
there exist numbers $\lambda_{1},\cdots,\lambda_{2^{n}-1}$ such that
\begin{displaymath}
  w=\sum\limits_{i=1}^{2^{n}-1}\lambda_{i}a_{i}~\text{and}~\sum\limits_{i=1}^{2^{n}-1}\lambda_{i}=1.
\end{displaymath}
Therefore,
\begin{align*}
\norm{w}&=\frac{1}{2^{n}}\norm{2^{n}w+v-v}\\
&=\frac{1}{2^{n}}\norm{w+(2^{n}-1)\left(w+\frac{1}{2^{n}-1}v\right)-v}\\
&\geq \frac{1}{2^{n}} 2^{n}=1,
\end{align*}
and the inequality holds since the distance from $v$ to ${\rm
  aff}F_{v}$ is $2^{n}$ and
\begin{displaymath}
  w+(2^{n}-1)\left(w+\frac{1}{2^{n}-1}v\right)=w+(2^{n}-1)\left(w-\frac{-1}{2^{n}-1}v\right)\in {\rm aff}F_{v}.
\end{displaymath}
Thus ${\rm aff}F_{v}$ is a supporting hyperplane of $\mball$. It
follows that $\mball\subseteq S$.
\par
4. Put 
\begin{displaymath}
  H_{1}=\{(\eta_{i})_{i=1}^{2^{n}-1}:~\eta_{2^{n-1}}=1\},~  H_{2}=\{(\eta_{i})_{i=1}^{2^{n}-1}:~\eta_{2^{n-1}}=-1\}.
\end{displaymath}
We have $\mball\subseteq S$. Thus, to show that $\thick{S}=2$, it
suffices to show that $H_{1}$ and $H_{2}$ are supporting hyperplanes of
$S$, which is clear since each of $H_{1}$ and $H_{2}$ supports both
$\mball$ and its polar body
\begin{displaymath}
  Q=\{(\eta_{i})_{i=1}^{2^{n}-1}:~\max\{|\eta_{i}|:~1\leq i\leq 2^{n}-1\}=1\},
\end{displaymath}
and $\mball\subseteq S\subseteq Q$.

\par
5. Let 
\begin{displaymath}
  H=\left\{(\eta_{i})_{i=1}^{2^{n}-1}:~\sum\limits_{i=1}^{2^{n}-1}\eta_{i}\leq
  1\right\}.
\end{displaymath}
Then $S\cap H$, a proper subset of $S$, is a convex body containing
$\mball$ whose thickness is $2$. Thus $S$ is not reduced.\qed
\end{proof}

Proposition \ref{prop:main} also shows that, for any positive number $\varepsilon $,
there exists a positive integer $n$, an $n$-dimensional Minkowski space
$\Mspace$, and a complete set $K$ in $\Mspace$ such that the ratio of
$\thick{K}$ to $\diam{K}$ is less than $\varepsilon
$. Anyway, since each non-trivial complete set (i.e., a complete set
containing at least two distinct points) in a Minkowski space has an interior point, its thickness is strictly greater
than 0.
\par
The situation in infinite dimensional Banach spaces is quite different:
it is possible that a non-trivial complete set is contained in a
hyperplane (cf. \cite[Example 4.6]{M-P-P}), and the thickness of such
a complete set is clearly 0.

\par
\bigskip 

We finish this paper with a very natural question, clearly not posed
in \cite{L-M2} and referring to all dimensions $n\geq 3$ (by
Meissner's theorem the planar case is clear; see \cite[p. 98]{M-S}).

\textbf{Problem:~}\textit{In which Minkowski spaces is any complete set also
reduced? Try to give geometric characterizations of their unit balls!}
\par

In some sense vice versa, it might be interesting to collect (e.g., as
partial results in that direction) typical geometric properties shared
by those Minkowski spaces which contain complete sets that are not
reduced. All this is directly motivated by our construction.

\par
More generally, it would be nice to give a complete clarification of
the containment and intersection relations between the families of
reduced, complete, and constant width bodies (and between suitable
subfamilies of them) in Minkowski Geometry, for example via a
completely described Venn diagram. This is of course closely related
to recent research on those Minkowski spaces in which completeness and
constant width are equivalent notions; see, e.g., \cite{M-Sch}.

%\begin{acknowledgements}
%\end{acknowledgements}

\bibliographystyle{spmpsci}      % mathematics and physical sciences
\bibliography{mbib}   % name your BibTeX data base

% \begin{thebibliography}{}
% etc
%\end{thebibliography}

\end{document}